\documentclass[a4paper,12pt]{amsart}

\usepackage[includehead, twoside=False]{geometry}

\usepackage{graphics}

\usepackage{amsmath,amssymb, amsthm,amsxtra}

\usepackage[mathscr]{eucal}

\usepackage{bm}
\usepackage[all]{xy}

\usepackage{aliascnt}

\theoremstyle{plain}
\newtheorem{thm}{Theorem}[section]
\newtheorem*{thm*}{Theorem}

\newaliascnt{prop}{thm}
\newaliascnt{cor}{thm}
\newaliascnt{lem}{thm}
\newaliascnt{claim}{thm}
\newaliascnt{defn}{thm}
\newaliascnt{ques}{thm}
\newaliascnt{conj}{thm}
\newaliascnt{fact}{thm}
\newaliascnt{rem}{thm}
\newaliascnt{ex}{thm}
\newtheorem{prop}[prop]{Proposition}
\newtheorem{cor}[cor]{Corollary}
\newtheorem{lem}[lem]{Lemma}
\newtheorem{claim}[claim]{Claim}
\newtheorem*{prop*}{Proposition}
\newtheorem*{cor*}{Corollary}
\newtheorem*{lem*}{Lemma}
\newtheorem*{claim*}{Claim}
\theoremstyle{definition}
\newtheorem{defn}[defn]{Definition}

\newtheorem*{defn*}{Definition}
\newtheorem*{ques*}{Question}
\newtheorem*{conj*}{Conjecture}

\newtheorem*{prob*}{Problem}

\newtheorem{rem}[rem]{Remark}
\newtheorem{ex}[ex]{Example}
\newtheorem*{fact*}{Fact}
\newtheorem*{rem*}{Remark}
\newtheorem*{ex*}{Example}
\aliascntresetthe{prop}
\aliascntresetthe{cor}
\aliascntresetthe{lem}
\aliascntresetthe{claim}
\aliascntresetthe{defn}
\aliascntresetthe{ques}
\aliascntresetthe{conj}
\aliascntresetthe{fact}
\aliascntresetthe{rem}
\aliascntresetthe{ex}

\usepackage[pointedenum]{paralist}
\usepackage{varioref}
\labelformat{equation}{\textnormal{(#1)}}
\labelformat{enumi}{\textnormal{(#1)}}

\setdefaultenum{(a)}{(i)}{(1)}{(A)}
\usepackage[a4paper]{hyperref}

\def\textsectionN~{\textsection{}}

\def\equationautorefname~#1\null{#1\null}
\usepackage{ifthen}
\usepackage{xspace}
\renewcommand\phi{\varphi}
\renewcommand\epsilon{\varepsilon}
\renewcommand\leq{\leqslant}
\renewcommand\geq{\geqslant}
\renewcommand\subsetneq{\subsetneqq}

\newcommand\num{\mathbb}\makeatletter
\newcommand{\set}{  \@ifstar{\@setstar}{\@set}}\newcommand{\@setstar}[2]{\{\, #1 \mid #2 \,\}}
\newcommand{\@set}[1]{\{ #1 \}}
\newcommand{\Set}[1]{\{\, #1 \,\}}
\makeatother
\newcommand{\lin}[1]{\langle\, #1 \,\rangle}
\newcommand{\pco}[1]{[\, #1 \,]}

\newcommand{\CC}{\num{C}}

\DeclareMathOperator{\codim}{codim}
\DeclareMathOperator{\Cone}{Cone}

\DeclareMathOperator{\Sing}{Sing}
\DeclareMathOperator{\Sec}{Sec}

\renewcommand{\Join}{\mathrm{Join}}

\newcommand{\A}{\mathbb{A}}
\newcommand{\Gr}{\mathbb{G}}
\newcommand{\TT}{\mathbb{T}}
\newcommand{\PP}{{\mathbb{P}}}

\newcommand{\PN}{\PP^N}
\newcommand\spcirc{^\circ}
\newcommand{\pr}{\mathrm{pr}}
\newcommand{\Pv}[1][N]{(\PP^{#1})\spcheck}
\newcommand{\fib}[1]{\overline{\gamma^{-1}(\gamma(#1))}}
\newcommand{\vspc}{\vspace{1.5ex}}

\title[Cubic hypersurfaces with positive dual defects]
{Cubic hypersurfaces with positive dual defects}

\author[K.~Furukawa]{Katsuhisa~FURUKAWA}
\address{Graduate School of Mathematical Sciences, 
the University of Tokyo, Tokyo, Japan}
\email{katu@ms.u-tokyo.ac.jp}

\begin{document}

\maketitle

\begin{abstract}
  We show that if a cubic hypersurface with positive dual defect
  over the complex number field
  is not a cone, then either the hypersurface coincides with
  the secant variety of the singular locus, or
  the hypersurface contains
  a linear subvariety of dimension greater than the dual defect
  such that the intersection of the singular locus and a general contact locus
  is contained in the linear subvariety.
\end{abstract}
\section{Introduction}

Let $X \subset \PN$ be a projective variety over the complex number field $\CC$.
The number $\delta_X := N-1-\dim(X^*)$ is called the \emph{dual defect} of $X$,
where $X^* \subset \Pv$ is the \emph{dual variety}, 
the variety of hyperplanes tangent to $X$.
Typical examples of $X$ with $\delta_X > 0$
are cones, joins, and secants of some varieties.
For a reference, see \cite{FP}.

Now let us consider a cubic hypersurface $X \subset \PN$ with $\delta_X > 0$.
Then $X$ is singular, and the secant variety of the singular locus,
\[
  \Sec(\Sing(X)) := \overline{\bigcup_{x,y \in  \Sing(X)} \lin{xy}} \subset \PN,
\]
is contained in $X$.
(This is because, if the line $\lin{xy} \subset \PN$ is not contained in  $X$ for singular points $x,y$,
then the length of $\lin{xy} \cap X$ is $\geq 4$, contrary to $\deg(X) = 3$.)

F.~L.~Zak \cite[IV, \textsection{}5]{Za} classified smooth projective varieties of codegree 3; in  particular,
if $X$ is a cubic hypersurface with $\delta_X > 0$
such that the deal variety $X^*$ is \emph{smooth}, then $X$ is one of the following;
(1) the $3$-fold in $\PP^4$ which is the image of
the composition of the Segre embedding $\PP^1 \times \PP^2 \hookrightarrow \PP^5$
and the linear projection $\PP^5 \dashrightarrow \PP^4$
from a point of $\PP^5 \setminus (\PP^1 \times \PP^2)$,
(2) the secant variety of a Severi variety,
or
(3) a general hyperplane section of the secant variety of a Severi variety.
Secants of Severi varieties are also known as the hypersurfaces
defined by homaloidal EKP cubic polynomials.
See \cite[Theorem~3]{Do}.

J.-M.~Hwang \cite{Hw} characterized secants of Severi varieties
as cubic hypersurfaces $X \subset \PN$
with nonzero Hessians having \emph{large} automorphism groups;
the condition is expressed in  the context of
the (first) prolongation $\mathfrak{aut}(\hat{X})^{(1)}$
of the Lie algebra $\mathfrak{aut}(\hat{X})$ of infinitesimal linear automorphisms.
For references, see also \cite{HM}, \cite{FH12}, \cite{FH18}.
We have $\delta_X > 0$
for a cubic hypersurface $X$ with nonzero Hessian and
$\mathfrak{aut}(\hat{X})^{(1)} \neq 0$,
due to \cite[Corollary~4.5]{Hw}.

On the other hand,
$\delta_X > 0$ holds for hypersurfaces $X \subset \PN$ with \emph{vanishing Hessians}.
U.~Perazzo \cite{Pe} investigated cubic hypersurfaces with vanishing Hessians,
which are not cones.
Such cubic hypersurfaces give examples of $X$ whose singular loci are linear varieties;
in particular, $\Sec(\Sing(X)) \neq X$.
R.~Gondim and F.~Russo \cite{GR} investigated
cubic hypersurfaces with vanishing Hessians after Perazzo,
and gave the classification for $N \leq 6$.
R.~Gondim, F.~Russo, and G.~Staglian\`{o} \cite{GRS} recently gave a general construction
of hypersurfaces with vanishing Hessians based on the dual Cayley trick.
For a reference, see also \cite[Chapter~7]{Ru}.

In this paper, we study cubic hypersurfaces with positive dual defects
even in the case when dual varieties are singular.
We denote by 
$\gamma: X \dashrightarrow X^* \subset \Pv$
the \emph{Gauss map} of $X$, which is a rational map sending
a smooth point $x \in  X$ to the embedded tangent space $\TT_xX \subset \PN$.
The fiber of $\gamma$ at $T \in X^*$
is called the contact locus on $X$ of $T$.
Our main result is:

\begin{thm}\label{mainthm}
  Let $X \subset \PN$ be a cubic hypersurface with $\delta_X > 0$
  and assume that $X$ is not a cone.
  Then exactly one of the following holds.
  \begin{enumerate}[\normalfont{}(I)]
  \item $X = \Sec(S)$ for an irreducible component $S$ of $\Sing(X)$.
  \item $X = \Join(Q_1, Q_2) := \overline{\bigcup_{x \in  Q_1, y \in  Q_2} \lin{xy}}$,
    where $Q_i$ is a smooth quadric hypersurface of
    $\lin{Q_i} = \PP^{\dim Q_i + 1} \subset \PN$ with $i=1,2$ such that
    $Q_1 \cap Q_2 = \lin{Q_1} \cap \lin{Q_2} = \set{z_0}$
    for a point $z_0 \in  \PN$; in this case, we have $\delta_X=1$.
  \item 
    There exists a linear variety $P \subset X$ of dimension $> \delta_X$
    such that $\fib{x} \cap \Sing(X) \subset P$ for general $x \in X$;
    in this case, $X \neq \Sec(\Sing(X))$.
  \end{enumerate}
\end{thm}

As we have seen above,
examples of (I) are given by the secant of a Severi variety
and a general hyperplane section of it.
More generally, we can also take
the intersection of the secant of a Severi variety and some possibly special hyperplanes
since $X^*$ do not need to be smooth.
Examples of (III) are given by cubic hypersurfaces with vanishing Hessians.
In the case of (II), the join $X$ has nonzero Hessian;
see \autoref{thm:calc-Join} for an explicit calculation of $X$.

\vspc

The paper is organized as follows.
In \autoref{sec:sing-locus-codim}, we calculate the defining equation of $X$ explicitly
in  the case when $\dim (\Sing(X)) = N-2$.
After that, we may assume $\dim (\Sing(X)) < N-2$.
From \autoref{sec:sing-points-gener} to \autoref{sec:cones-codim-one},
we investigate the case when the dual defect $\delta_X=1$.
Then the fiber $F = \fib{x} \subset X$ of the Gauss map $\gamma$ at a general point $x \in  X$ is a \emph{line} intersecting $\Sing(X)$ (see \autoref{thm:FnSing-codim1}).

In \autoref{sec:sing-points-gener},
we consider the locus of intersection points $\fib{x} \cap \Sing(X)$ with general $x \in X$,
and take $Z_1, \dots, Z_{\kappa}$ to be the irreducible components of the closure of the locus.
Focusing on a relationship between $F$ and a point $w \in  Z_k \cap \TT_xX$,
we give a key technical result describing fibers of $\gamma$ near $F$
(\autoref{thm:M-in-Xz}).
In \autoref{sec:join-two-varieties},
we characterize cubic hypersurfaces which are \emph{joins} of subvarieties
(\autoref{thm:char-Xjoin}).
This implies that if $X = \Sec(\Sing(X))$, then either (I) or (II) of \autoref{mainthm} holds
(see \autoref{thm:del1-I-II}).

In \autoref{sec:secant-sing-locus}, in order to discuss the condition~(III) of \autoref{mainthm},
we study the case when $\Sec(\Sing(X)) \neq X$ and $\delta_X = 1$.
Then $\kappa = 1$ and that we write $Z = Z_1$.
The goal is to show that the linear subvariety $\lin{Z} \subset \PN$ spanned by $Z$
is indeed contained in $X$ (\autoref{thm:large-lin-Z}).
We argue by contradiction, and suppose $\lin{Z} \not\subset X$.
Then, in \autoref{sec:gauss-map-cone},
we show $\Cone_w(Z) \not\subset \Sing(X)$ for general $w \in Z$ (\autoref{thm:dim-gamma-ConezZ}).
In \autoref{sec:cones-codim-one},
cutting with tangent hyperplanes $T \in  \gamma(\Cone_w(Z))$,
we analyze the structure of an \emph{$(N-2)$-dimensional cone} $C_w \subset X$.
As a result, \autoref{thm:large-lin-Z} is proved.

In \autoref{sec:induct-dual-defects}, we study the case of any $\delta_X > 0$.
From the case of $\delta_X = 1$ and by induction on $\delta_X$,
we show a slightly weaker version of the main result (\autoref{mainthm-b}).
Also by induction, it holds that $\lin{Z}$ is not a $\delta_X$-plane if $\Sec(\Sing(X)) \neq X$.
Then we finish the proof of \autoref{mainthm}.

\section{Singular points, dual defects, and Gauss maps}

Let $X \subset \PN$ be  a cubic hypersurface, and let
$\gamma: X \dashrightarrow X^* \subset \Pv$ be the Gauss map defined by
$\gamma(x) = \TT_xX \in  \Pv$.
We consider the closure $F \subset X$ of a general fiber of $\gamma$.
Note that $\dim F = N-1-\dim(X^*) = \delta_X$.
It is known that $F$ is {a linear variety},
in particular, is {irreducible}. See \cite{FP}, \cite[I, Theorem~2.3(c)]{Za}.

\begin{rem}\label{thm:FnSing-codim1}
  From \cite[I, Theorem~1.7]{Za},
  it holds that
  \[
    \codim(F \cap \Sing(X), F) = 1
  \]
  for the closure $F \subset X$
  of a general fiber of $\gamma$.
  In particular,
  $X$ must be a cone
  if $\dim (\Sing (X)) < \delta_X$.
\end{rem}

\subsection{Singular locus of codimension one}
\label{sec:sing-locus-codim}

In this subsection, we study the case when
the dimension of $\Sing (X)$ is maximum.

\begin{prop}\label{thm:singX-N-2}
  Let $X \subset \PN$ be a cubic hypersurface which is not a cone.
  Assume $\dim(\Sing(X)) = N-2$. Then $\delta_X > 0$ if and only if
  $(N, \delta_X) = (4, 1)$ and $X$ is
  projectively equivalent to $(x_0x_1 x_2 + x_0^2 x_4 + x_1^2 x_3 = 0) \subset \PP^4$;
  in  this case, the singular locus is a $2$-plane.
\end{prop}

The above $3$-fold in $\PP^4$ is the image of $\PP^1 \times \PP^2$
under the composite morphism of the Segre embedding $\PP^1 \times \PP^2 \hookrightarrow \PP^5$
and the linear projection $\PP^5 \dashrightarrow \PP^4$
from a point of $\PP^5 \setminus (\PP^1 \times \PP^2)$,
which is the dual variety of (I) of \cite[IV, Theorem 5.2]{Za}.
See also \cite[2.2.9]{FP}.
It is known that this $3$-fold is a hypersurface with vanishing Hessian.
For more details, see \cite[Theorem~4.2]{GR}, \cite[Theorem~7.6.7]{Ru}.
In order to prove the proposition,
let us show the following two lemmas.

\begin{lem}\label{thm:singX-N-2-lem}
  The statement of \autoref{thm:singX-N-2} holds
  if there exists an
  $(N-2)$-plane $M$ of $\PN$ contained in  $\Sing(X)$.
\end{lem}
\begin{proof}
  Choosing the homogeneous coordinates $\pco{x_0: x_1: \dots: x_N}$ on $\PN$,
  we my assume $M = (x_0 = x_1 = 0)$. Then
  the defining equation of $X$ is written as
  $F = x_0f + x_1g$ with some homogeneous polynomials $f,g$ of degree $2$.
  We write $F_{x_i} = \partial F/\partial x_i$, and so on.
  Since $M \subset \Sing(X)$, we have
  $F_{x_i}|_M = 0$.
  Since $F_{x_0} = f + x_0f_{x_0} + x_1g_{x_0}$,
  we have $f|_M = 0$; hence $f = x_0 f' + x_1 f''$
  with some linear homogeneous polynomials $f'$ and $f''$.
  In the same way, $g = x_0 g' + x_1 g''$.
  Therefore, we can write $F = x_0x_1 l_1 + x_0^2 l_2 + x_1^2 l_3$
  with linear homogeneous polynomials $l_1,l_2,l_3 \in  k[x_0, \dots, x_N]$.
  
  Since $X$ is not a cone,
  and since $F$ is generated only by
  $5$ linear polynomials $x_0,x_1,l_1,l_2,l_3$,
  we have $N \leq 4$; moreover,
  changing the homogeneous coordinates on $\PN$,
  we have
  \[
    F = { x_0x_1 x_2 + x_0^2 x_4 + x_1^2 x_3}
  \]
  if $N = 4$,
  and
  \[
    F = x_0x_1 x_2 + x_0^2 x_3 + x_1^2l
    \,\text{ or }
    F = x_0x_1 l + x_0^2 x_3 + x_1^2 x_2
  \]
  with a linear polynomial $l \in  \CC[x_0, x_1, x_2, x_3]$
  if $N = 3$.
  In the case of $N = 4$,
  it holds $\delta_X = 1$; for an explicit calculation, see \autoref{thm:ex-calc-Z} below.
  In the case of $N = 3$, we can also check $\delta_X = 0$ directly.

  For example, we consider the case of $N=3$ and $F = {x_0x_1 x_2 + x_0^2 x_3 + x_1^2l}$.
  Taking $x_0 = 1$ and $x_3 = \phi := -x_1 x_2 - x_1^2l$,
  we have that
  $\iota: \A^2 \hookrightarrow X \subset \PP^3$
  is defined by
  $(x_1,x_2) \mapsto (1, x_1,x_2,\phi)$.
  Then $\TT_{\iota(x)}X$ is equal to the $2$-plane
  spanned by the $3$ points corresponding to the row vectors of
  the matrix,
  \[
    \begin{bmatrix}
      1 & x_1 &x_2 & \phi
      \\
      0 &1 &0 & \phi_{x_1}
      \\
      0 & 0 &1 & \phi_{x_2}
    \end{bmatrix}.
  \]
  This is transformed to the following matrix under elementary row operations,
  \[
    \begin{bmatrix}
      1 & 0 & 0 & \phi-x_1\phi_{x_1}-x_2\phi_{x_2}
      \\
      0 &1 &0 & \phi_{x_1}
      \\
      0 & 0 &1 & \phi_{x_2}
    \end{bmatrix}.
  \]
  Then
  $\gamma \circ \iota: \A^2 \hookrightarrow X \dashrightarrow \Pv[3]$
  is given by
  \begin{equation}\label{eq:param-gamma}
    (x_1,x_2) \mapsto
    (\phi-x_1\phi_{x_1}-x_2\phi_{x_2}, \phi_{x_1}, \phi_{x_2})
    \in  \A^3 \subset \Pv[3].
  \end{equation}
  Since
  \[
    \begin{bmatrix}
      (\phi-x_1\phi_{x_1}-x_2\phi_{x_2})_{x_1}, \phi_{x_1x_1}, \phi_{x_2x_1}) \\
      (\phi-x_1\phi_{x_1}-x_2\phi_{x_2})_{x_2}, \phi_{x_1x_2}, \phi_{x_2x_2})
    \end{bmatrix}
    =
    \begin{bmatrix}
      * & * & -1 - 2x_1l_{x_2} \\
      * & -1 - 2x_1l_{x_2} & 0
    \end{bmatrix}
  \]
  with $l_{x_2} = \partial l / \partial x_2 \in  \CC$,
  and since $(-1 - 2x_1l_{x_2})^2 \neq 0$ for general $x \in  \A^2$,
  the rank of the linear map
  $d_{\iota(x)} \gamma: \TT_{\iota(x)}X \rightarrow \TT_{\gamma(\iota(x))}\Pv[3]$
  between Zariski tangent spaces
  is equal to $2$.
  Hence $\dim (\gamma(X)) = 2$, i.e., $\delta_X=0$.
\end{proof}

\begin{lem}\label{thm:sec-linear}
  Let $S \subset \PN$ be a projective variety
  such that $\dim \Sec(S) = \dim S + 1$.
  Then $\Sec(S)$ is a linear variety.
\end{lem}
\begin{proof}
  For general $x \in  S$, we have $\Sec(S) = \Cone_x(S)$.
  Fix a general point $x_0 \in  S$, and consider the linear projection
  $\pi = \pi_{x_0}: \PN \setminus \set{x_0} \rightarrow \PP^{N-1}$.
  Set
  $S' = \pi(S) = \pi(\Cone_{x_0}(S)) = \pi(\Sec(S))
  \subset \PP^{N-1}$.

  For general $x' = \pi(x) \in  S'$ with $x \in  S$,
  we have
  $\Cone_{x'}(S') = \pi(\Cone_{x}(S)) = \pi(\Sec(S)) = S'$.
  This means that $S'$ is a linear variety.
  Hence the preimage of $S'$ under $\pi$,
  which coincides with $\Sec(S)$, is a linear variety.
\end{proof}

\begin{proof}[Proof of \autoref{thm:singX-N-2}]
  Let $X$ be a cubic hypersurface which is not a cone.
  Let $S \subset \Sing(X)$ be an irreducible component
  such that $\dim S = N-2$.
  If $\Sec(S) = X$, then it follows from \autoref{thm:sec-linear}
  that $X$ must be a hyperplane, a contradiction.
  Thus $\Sec(S) = S$, i.e., $S$ is an $(N-2)$-plane.
  From \autoref{thm:singX-N-2-lem},
  we have the result.
\end{proof}

\begin{rem}\label{thm:N3}
  Let $X \subset \PP^3$ be a cubic surface with $\delta_X > 0$.
  Then $X$ is a cone.
  This is because, it follows from \autoref{thm:singX-N-2} that $\dim (\Sing(X)) = 0$;
  then $X$ is a cone as in  \autoref{thm:FnSing-codim1}.
\end{rem}

\begin{ex}\label{thm:ex-calc-Z}
  Let $X \subset \PP^4$ be
  the cubic $3$-fold defined by
  $(x_0x_1 x_2 + x_0^2 x_4 + x_1^2 x_3 = 0)$
  as in  \autoref{thm:singX-N-2}.
  We consider the locus $Z \subset \Sing(X)$
  of intersection points ${\fib{x} \cap \Sing(X)}$ with general $x \in X$.
  Then $Z$ is the conic $(x_2^2 - 4 x_3x_4  = 0)$ in the $2$-plane
  $\Sing(X) = (x_0 = x_1 = 0) \subset \PP^4$, as follows.

  Taking $x_0 = 1$, 
  we have a morphism
  $\iota: \A^{3} \hookrightarrow X \subset \PN$ sending
  \[
    (x_1, \dots, x_{3}) \mapsto \pco{1: x_1: \dots: x_{3}: \phi}
    \ \text{ with }\ \phi := - x_1 x_2 - x_1^2 x_3.
  \]
  In the same way as the formula~\autoref{eq:param-gamma} in the proof of \autoref{thm:singX-N-2-lem},
  the Gauss map
  $\gamma\circ \iota: \A^{3}\rightarrow \A^{4} \subset \Pv[4]$
  is described by
  \begin{multline*}
    (x_1, \dots, x_{3})
    \mapsto
    \\
    (\phi - \sum x_i \phi_{x_i}, \phi_{x_1}, \dots, \phi_{x_{3}})
    =
    (x_1x_2+2x_1^2x_3, -x_2-2x_1x_3, -x_1, -x_1^2).
  \end{multline*}
  Take general numbers $a,b \in \CC$ and consider the formula $b = -x_2-2ax_3$.
  Setting $\alpha = (-ab, b,-a,-a^2) \in \A^4$,
  we have 
  $(\gamma\circ \iota) (a, -b-2ax_3, x_3) = \alpha$ for each $x_3 \in \CC$.
  Taking $u=x_3$, we have that the fiber $\overline{\gamma^{-1}(\alpha)} \subset X$
  is the line parameterized by the morphism $\PP^1 \rightarrow \PP^4$,
  \[
    \pco{t:u} \mapsto \pco{t: at: -bt-2au, u: abt+a^2u}.
  \]
  Therefore
  $\overline{\gamma^{-1}(\alpha)} \cap \Sing(X) = \pco{0: 0: -2a: 1: a^2}$.
  Hence $Z$ is the conic $(x_2^2 - 4 x_3x_4  = 0) \subset \PP^2$.
\end{ex}

\subsection{Singular points and general fibers of the Gauss map}
\label{sec:sing-points-gener}

The goal of this subsection is to show \autoref{thm:M-in-Xz},
which plays important roles in  later discussions.

Let $X \subset \PN$ be a cubic hypersurface which is not a cone.
Assume $\delta_X = 1$. Then the closure $F \subset X$ of a general fiber of the Gauss map $\gamma$
is a line intersecting $\Sing(X)$ (see \autoref{thm:FnSing-codim1}).
We consider irreducible closed subsets $Z_1, \dots, Z_{\kappa} \subset \Sing(X)$ as follows.

\begin{defn}\label{thm:def-Z}
  Let $X\spcirc \subset X$ be a non-empty open subset consisting of
  smooth points $x$ of $X$ such that $\fib{x}$ is a line.
  We set
  \[
    \hat{\mathscr{Z}} =
    \set*{(x, z) \in  X\spcirc \times \Sing(X)}{z \in  \fib{x}}
  \]
  with the projection $\pr_i$ from $\hat{\mathscr{Z}}$ to the $i$-th factor ($i=1,2$).
  Let $\mathscr{Z}_1, \dots, \mathscr{Z}_{\kappa}$ be
  the irreducible components of $\hat{\mathscr{Z}}$
  with an integer $\kappa \geq 1$.
  By shrinking $X\spcirc$ if necessarily,
  we may assume that
  $\pr_1|_{\mathscr{Z}_k}: \mathscr{Z}_k \rightarrow X\spcirc$
  is dominant for each $1 \leq k \leq \kappa$.
  Since $\pr_1|_{\mathscr{Z}_k}$ is generically finite,
  it follows $\dim \mathscr{Z}_k = N-1$.

  We set $Z_k \subset \Sing(X)$ to be the closure of the image of
  $\pr_2|_{\mathscr{Z}_k}: \mathscr{Z}_k \rightarrow \Sing(X)$.
  For general $x \in  X$ and for each $k$,
  the intersection $\fib{x} \cap Z_k \neq \emptyset$ and
  \begin{equation*}
    \fib{x} \cap \Sing(X) = \bigcup_{1 \leq k \leq \kappa} \fib{x} \cap Z_k.
  \end{equation*}
  For general $z \in  Z_k$,
  we set
  \[
    X_{z} = \overline{\pr_1(\pr_2^{-1}(z))} \subset X,
  \]
  which consists of general fibers of $\gamma$ passing through $z$;
  in  particular, $X_{z}$ is a cone with vertex $z$.
  Since $z \notin Z_{k'}$ for all $k' \neq k$,
  each irreducible component of $X_{z}$ is of dimension $N-1-\dim Z_k$.
\end{defn}

\begin{rem}
  We have $\dim Z_k > 0$ since $X$ is not a cone.
  If $\Sec(Z_k) \neq X$, then
  $\fib{x} \cap Z_k$ is the set of a point for general $x \in  X$.
\end{rem}
\begin{rem}\label{thm:xi-2-case}
  It holds $\kappa = 1$ if $X \neq \Sec(\Sing(X))$.
  This is because, in the case of $\kappa \geq 2$,
  since the closure of a general fiber of $\gamma$ intersects
  $Z_k$ for each $1 \leq k \leq \kappa$,
  we have $X = \Join(Z_{k_1}, Z_{k_2})$ for $k_1 \neq k_2$ and then $X = \Sec(\Sing(X))$.

  An explicit calculation of $Z = Z_1$ with $X \neq \Sec(\Sing(X))$
  have been seen in \autoref{thm:ex-calc-Z}.
\end{rem}

We often use the following description of lines and singular points of $X$.

\begin{lem}\label{thm:3rd-line}
  Let $s_1, s_2 \in  \Sing(X)$ and let $l_1 \subset X$ be a line
  such that $s_1 \in  l_1$ and $s_2 \notin l_1$.
  Set $M = \lin{l_1s_2} \subset \PN$, the $2$-plane spanned by $l_1$ and $s_2$.
  If $M \not\subset X$,
  then one of the following holds.
  \begin{enumerate}[\noindent\normalfont{}(i)]
  \item $\Sing(X \cap M)$ is the line $\lin{s_1s_2}$, and
    $X \cap M$ is set-theoretically equal to $\lin{s_1s_2} \cup l_1$.
  \item $\Sing(X \cap M)$ is the set of $3$ points $\Set{s_1, s_2, l_1\cap l_2}$,
    and $X \cap M$ is equal to $\lin{s_1s_2} \cup l_1 \cup l_2$,
    where $l_2 \subset M$ is a line
    such that $s_2 \in  l_2$ and $s_1 \notin l_2$.
  \end{enumerate}
\end{lem}
\begin{proof}
  We have $\lin{s_1s_2} \cup l_1 \subset X \cap M$.
  Assume that (i) does not hold.
  Since $\deg (X \cap M) = 3$ and $s_2 \in  \Sing(X \cap M)$,
  there exists a line
  $l_2 \subset X \cap M$ passing through $s_2$ such that $l_2 \neq \lin{s_1s_2}$.
  Then the condition (ii) holds.
\end{proof}

\begin{lem}\label{thm:Cw-def-lem}
  Let $x \in  X$ be a general point such that
  $F:= \fib{x} \subset X$ is a line.
  Let $w \in  \Sing(X) \cap \TT_xX \setminus F$.
  Then $\lin{xw} \subset \lin{Fw} \subset X$.
\end{lem}

\begin{proof}
  Assume
  $M := \lin{Fw} \not\subset X$,
  and take $z \in F \cap \Sing(X)$.
  Then $w,z \in  \Sing(X \cap M)$.
  Since $M \subset \TT_xX = \TT_{\tilde{x}}X$ for any
  $\tilde{x} \in  F \setminus \Sing(X)$,
  we have $F \subset \Sing(X \cap M)$,
  which contradicts \autoref{thm:3rd-line}.
\end{proof}

Note that $\Sing(X) \not\subset \TT_xX$ for general $x \in  X$ since $X$ is not a cone.
We next consider the cone $X_z \subset X$
given in  \autoref{thm:def-Z}.

\begin{lem}\label{thm:M-in-Xz-lem}
  In the setting of \autoref{thm:Cw-def-lem},
  take a point $z \in  F \cap Z_1$,
  a general point $w \in  Z_1 \cap \TT_xX \setminus F$,
  and the $2$-plane $M := \lin{Fw} \subset X$.
  Let $\tilde{x} \in  M \setminus (\Sing(X) \cup \lin{wz} \cup F)$
  and assume that the line
  $\lin{\tilde{x}z} \subset M$
  is contracted to a point under $\gamma$.
  Then
  $M \subset X_z$.
\end{lem}

Before proving the above lemma, we prepare a notation.

\begin{defn}\label{thm:Cw}
  Let $X\spcirc \subset X$ be a non-empty open subset
  consisting of smooth points of $X$.
  Then we set
  \[
    \mathscr{\hat{C}} = \set*{(x, w) \in  X\spcirc \times Z_1}{w \in  \TT_xX}
  \]
  with the projection $\pr_i$ from $\mathscr{\hat{C}}$ to the $i$-th factor ($i = 1, 2$).
  By shrinking $X\spcirc$ if necessarily,
  we may assume that $\pr_1|_{\mathscr{C}}: \mathscr{C} \rightarrow X\spcirc$ is dominant 
  for each irreducible component $\mathscr{C} \subset \hat{\mathscr{C}}$.
  Considering a general $(x, w) \in  \mathscr{C}$ which is not contained in 
  any other irreducible components of $\mathscr{\hat{C}}$,
  since the fiber of $\pr_1$ at $x$ is isomorphic to $Z_1 \cap \TT_xX$,
  we have $\dim \mathscr{C} = N+\dim(Z_1)-2$.

  For general $w \in  Z_1$,
  we set $C_w \spcirc \subset X$ to be an irreducible component of
  \[
    \pr_1(\pr_2^{-1}(w)) = \set*{x \in  X\spcirc}{w \in  \TT_xX}.
  \]
  It follows from \autoref{thm:Cw-def-lem}
  that
  $C_w := \overline{C_w\spcirc} \subset X$ is a cone with vertex $w$,
  where $\dim C_w = N-2$.
  Then $X$ is \emph{swept out} by $C_w$'s with general $w \in  Z_1$,
  that is, $X$ is equal to the closure of the union of $C_w$'s with general $w \in  Z_1$.
\end{defn}

For a linear variety $A \subset \PN$,
we denote by $A^* \subset \Pv$ the set of hyperplanes containing $A$.
Then $A^*$ is an $(N-1-\dim A)$-plane of $\Pv$.

\begin{rem}
  We have $\dim (\gamma(C_w)) = N-3$;
  in  particular, $C_w$ is swept out by lines $\fib{u}$
  with general $u \in  C_w$.

  The reason is as follows.
  Since $\dim X^* = N-2$ and $\gamma(C_w) \subset X^* \cap w^*$,
  we have $\dim \gamma(C_w) \leq N-3$.
  Since $C_w$'s sweep $X$, a general point $u \in  C_w$ is general in  $X$, that is,
  $\fib{u}$ is a line. Considering $\gamma|_{C_w}$, we have
  $\fib{u} \subset C_w$ and $\dim (\gamma(C_w)) = N-3$.
\end{rem}

\begin{proof}[Proof of \autoref{thm:M-in-Xz-lem}]
  We may take $(x, w)$ as a general pair in $\mathscr{\hat{C}}$, and set $F = \fib{x} \subset X$.
  Then $M \subset C_w$.
  Set $\tilde{F} = \lin{\tilde{x}z} \subset M$.
  By assumption, we have $\tilde{F} \neq \lin{wz}$ and $\tilde{F} \neq F$.
  Let $l_1, l_2 \subset M$ be two general lines passing through $w$,
  and fix a general point $\alpha_i \in  l_i$ for $i=1,2$.
  Then $\gamma(l_i) \subset (\TT_{\alpha_i}C_w)^*$ in  $\Pv$.
  (This is because,
  since $C_w$ is a cone with vertex $w$,
  $\TT_{\alpha_i}C_w = \TT_{y}C_w \subset \TT_{y}X$
  for general $y \in  l_i$.)
  Since $\dim C_w = N-2$, we have $\dim ((\TT_{\alpha_i}C_w)^*) = 1$,
  which implies that $\gamma(l_i)$ is indeed the line
  $(\TT_{\alpha_i}C_w)^*$.
  On the other hand, since $M = \PP^2$,
  we have $l_i \cap F \neq \emptyset$ and $l_i \cap \tilde{F} \neq \emptyset$.
  Then $\gamma(l_i) = \lin{\gamma(F)\gamma(\tilde{F})}$ in  $\Pv$.
  Therefore, $\gamma(l_1) = \gamma(l_2)$. By the generality of choices of $l_1,l_2$,
  we have $\dim \gamma(M) = 1$.

  Since $\dim \gamma(M) = 1$, the line $E := \fib{y}$ is contained in  $M$
  for general $y \in  M$. Then
  $F \cap E = \set{s}$ and
  $\tilde{F} \cap E = \set{\tilde{s}}$
  with some $s, \tilde{s} \in  \Sing(X)$.
  We have $s = \tilde{s}$
  (otherwise, $y \in  E = \lin{s\tilde{s}}$,
  contradicting the generality of $y$),
  and then the point is equal to the intersection point $z$
  of $F$ and $\tilde{F}$.
  It follows $y \in  {X}_z$.
  Therefore
  $M \subset {X}_z$.
\end{proof}

Now let us show the following key technical result,
where $Z_1 \subset \Sing(X)$ and $X_z \subset X$ are as in  \autoref{thm:def-Z}.
\begin{prop}\label{thm:M-in-Xz}
  Let $X \subset \PN$ be a cubic hypersurface with $\delta_X = 1$, which is not a cone.
  Assume that $Z_1$ is non-linear and $\Sec(Z_1) \neq X$.
  For a general point $z \in Z_1$,
  assume that $Z_1$ is swept out by $Z_1 \cap \TT_xX$'s with general $x \in  X_z$.
  Let $F \subset X$ be the closure of a general fiber of $\gamma$
  passing through $z$.
  For a smooth point $x \in  F$ and for general $w \in  Z_1 \cap \TT_xX$, we set
  $M = \lin{wF}$, which is a $2$-plane contained in  $X$
  as in  \autoref{thm:Cw-def-lem}.
  Then we have $M \subset X_z$.
  (Hence $M$ consists of fibers of $\gamma$ near $F$ passing through $z$.)
\end{prop}

Note that, though $Z_1$ is swept out by $Z_1 \cap \TT_xX$'s with general $x$ in  $X$,
the same statement does \emph{not} always hold with general $x$ in  ``$X_z$''
(see \autoref{rem:join-two-varieties}).
Thus the assumption in  the proposition is necessary.

\begin{proof}[Proof of \autoref{thm:M-in-Xz}]
  Let $z \in  Z_1$ be general.
  For general $x \in  X_z$, we consider
  \begin{equation}\label{eq:large-cone-in-X}
    \Cone_{\lin{xz}}(Z_1 \cap \TT_xX) \subset \PN,
  \end{equation}
  which is a cone with vertex $\lin{xz}$,
  and is in  fact contained in  $X$
  by \autoref{thm:Cw-def-lem}
  since it is the closure of the union of
  $2$-planes $\lin{wxz}$'s
  with general $w \in  Z_1 \cap \TT_xX$.
  \begin{claim}
    The above cone~\ref{eq:large-cone-in-X} is of dimension $\dim(Z_1)+1$.
  \end{claim}
  \begin{proof}
    Set $r = \dim(Z_1)$.
    For general $x \in  X_z$,
    the intersection $Z_1 \cap \TT_xX$ is 
    not a cone with vertex $z$.
    Otherwise, since $Z_1$ is swept out by $Z_1 \cap \TT_xX$'s with general $x \in  X_z$,
    the locus $Z_1$ is itself a cone with vertex $z$; by the generality of $z$,
    it means that $Z_1$ is a linear variety, a contradiction.

    Thus $\dim (\Cone_z(Z_1 \cap \TT_xX)) = r$.
    Since $\Sec(Z_1) \neq X$, we have $X_z \not\subset \Cone_z(Z_1)$.
    Thus $x \notin \Cone_z(Z_1)$, which implies
    that the cone~\ref{eq:large-cone-in-X} is of dimension $r+1$.
  \end{proof}

  Let $X\spcirc \subset X$ be a non-empty open subset consisting of smooth points,
  and let $H \subset \PN$ be a general hyperplane such that $z \notin H$.
  Then $\gamma(X_z) = \gamma(X_z \cap H)$
  and $X_z = \Cone_z (X_z \cap H)$.
  We consider
  \[
    Y := \overline{\bigcup_{x \in  X_z\spcirc \cap H} \Cone_{\lin{xz}}(Z_1 \cap \TT_xX)} \subset X.
  \]
  In addition, we set
  \[
    I = \set*{(x,p,w) \in  (X_z\spcirc \cap H) \times X \times Z_1}
    {w \in  Z_1 \cap \TT_xX \setminus \lin{xz}\text{ and } p \in  \lin{wxz}}
  \]
  and
  \[
    J = \set*{(x,p) \in  (X_z\spcirc \cap H) \times X}
    {p \in  \Cone_{\lin{xz}}(Z_1 \cap \TT_xX)},
  \]
  with the projection
  $\rho: I \rightarrow J$ sending $(x,p,w) \mapsto (x,p)$
  and the projection
  $\phi: J \rightarrow Y$ sending $(x,p) \mapsto p$.

  Since the fiber of
  $J \rightarrow X_z \cap H: (x,p) \mapsto x$
  at $x \in  X_z \cap H$
  corresponds to the cone~\ref{eq:large-cone-in-X},
  and since $\dim X_z = N-1-\dim Z_1$,
  we have $\dim J = N-1$.
  Since $X$ is not a cone, we have $\dim Y < N-1$.
  Thus each irreducible component of
  a fiber of $J \rightarrow \overline{\phi(J)} = Y$ is of dimension $\geq 1$.

  Let $x \in  X_z$ be a general point.
  For general $w \in  Z_1 \cap \TT_xX$,
  we take $A \subset I$ to be an irreducible component of $\rho^{-1}(\phi^{-1}(x))$
  such that $(x,x,w) \in  A$.
  Then the image $\rho(A) \subset J$ is an irreducible component of $\phi^{-1}(x)$
  such that $(x,x) \in  \rho(A)$, where $\dim \rho(A) \geq 1$.
  In particular, for a general point $(\tilde{x},x,\tilde{w}) \in  A$,
  we have $\tilde{x} \neq x$.
  Let $\tilde{M} := \lin{\tilde{w}xz} = \lin{\tilde{w}\tilde{x}z} \subset X$.
  From \autoref{thm:M-in-Xz-lem},
  since two lines
  $\lin{xz}, \lin{\tilde{x}z} \subset \tilde{M}$
  are contracted to points under $\gamma$,
  we have 
  $\tilde{M} \subset X_z$.
  Set $A' \subset Z_1$ to be the closure of the image of $A$ under $(x,p,w) \mapsto w$.
  Then $\Cone_{\lin{xz}}(A') \subset X_z$. Since
  $M = \lin{wxz} \subset \Cone_{\lin{xz}}(A')$, the assertion follows.
\end{proof}

\subsection{Join of two varieties}
\label{sec:join-two-varieties}

In this subsection,
using techniques given in  \autoref{sec:sing-points-gener},
we study a cubic hypersurface $X \subset \PN$
which is the join of two varieties.
Our aim is to give the following characterization.

\begin{thm}\label{thm:char-Xjoin}
  Let $X \subset \PN$ be a cubic hypersurface
  such that $X = \Join(Z_1, Z_2)$ for
  two irreducible components $Z_1, Z_2$ of $\Sing(X)$.
  Assume $\delta_X=1$ and assume that $X$ is not a cone.
  Then $Z_i$ is a smooth quadric hypersurface
  of $\lin{Z_i} = \PP^{\dim(Z_i)+1}$ for $i=1,2$
  such that $Z_1 \cap Z_2 = \lin{Z_1} \cap \lin{Z_2} = \set{z_0}$
  for a point $z_0$.
\end{thm}

In order to use \autoref{thm:M-in-Xz}, we need to show a lemma.

\begin{lem}\label{thm:join-sweepZ}
  Let $X = \Join(Z_1, Z_2) \subset \PN$ be as in  \autoref{thm:char-Xjoin}.
  Assume $\dim (\lin{Z_1}) > \dim Z_1+1$,
  and let $z \in  Z_1$ be general.
  Then $Z_1$ is swept out by $Z_1 \cap \TT_{x}X$'s
  with general $x \in  X_z=\Cone_z(Z_2)$.
\end{lem}
\begin{proof}
  Since $X$ is not a cone, the locus $Z_1$ is not a linear variety.
  Set $r = \dim Z_1$.
  Set $K = \TT_zZ_1$, an $r$-plane of $\PN$.
  For general $x \in  X_z$,
  we have $K \subset \TT_xX$.
  This is because,
  since $x \in  \lin{z\tilde{z}}$ with some $\tilde{z} \in  Z_2$,
  it follows from Terracini's lemma that
  $\TT_{x}X = \lin{\TT_{z}Z_1, \TT_{\tilde{z}}Z_2}$.
  Hence $\gamma(X_z) \subset K^*$ in  $\Pv$,
  where 
  $\dim(\gamma(X_z)) = N-r-2$ and $K^* \simeq \PP^{N-r-1}$.

  \begin{claim}
    $\gamma(X_z) \subset K^*$ is not an $(N-r-2)$-plane of $\Pv$
  \end{claim}
  \begin{proof}
    Otherwise, we have $\gamma(X_z) = L^*$
    for some $(r+1)$-plane $L \subset \PN$ containing $K$.
    Since
    $Y := \Cone_{K}(Z_2) \subset \PN$ is a hypersurface
    such that
    $Y^* = \gamma(X_z) = L^*$,
    it follows that $Y = Y^{**} = L$ is indeed a hyperplane.
    Then $\gamma(X_z)$ is a point, yielding $r=N-2$.
    Since $X = \Join(Z_1, Z_2) = \Sec(\Sing(X))$ by assumption,
    this contradicts \autoref{thm:singX-N-2}.
  \end{proof}

  Take a general point $z' \in  Z_1$. By assumption,
  $Z_1 \not\subset \lin{z'K}$.
  Hence we have
  $\lin{z'K} \neq \lin{z''K}$
  for general $z', z'' \in  Z_1$.
  Let $\mathscr{L}\spcirc$ be the set of
  $\lin{z'K} \in  \Gr(r+1, \PN)$ with general $z' \in  Z_1$.
  Then $\dim \mathscr{L}\spcirc \geq 1$.
  
  For a curve $\mathscr{L}_1\spcirc \subset \mathscr{L}\spcirc$,
  and for the closure
  $\mathscr{L}_1 \subset \Gr(r+1, \PN)$ of $\mathscr{L}_1\spcirc$,
  since $K^* = \bigcup_{L \in  \mathscr{L}_1} L^*$ in  $\Pv$,
  and since $\gamma(X_z) \subset K^*$
  is not contained in  the finite union
  $\bigcup_{L \in  \mathscr{L}_1 \setminus \mathscr{L}_1\spcirc} L^*$,
  a general point $\gamma(x) \in  \gamma(X_z)$
  is contained in  $L^*$ for $L = \lin{z'K} \in  \mathscr{L}_1$
  with some general $z' \in  Z_1$, that is,
  $z' \in  \TT_xX$.
\end{proof}

\begin{proof}[Proof of \autoref{thm:char-Xjoin}]
  First we show
  $\dim (\lin{Z_i}) = \dim(Z_i) + 1$ and show that
  $Z_i$ is a quadric hypersurface
  of $\lin{Z_i} = \PP^{\dim(Z_i)+1}$ with $i=1,2$.
  It is sufficient to show the statement for $i=1$.
  Since $Z_1$ is non-linear,
  we have $\dim (\lin{Z_1}) \geq \dim Z_1 + 1$.

  Suppose that $\dim (\lin{Z_1}) > \dim Z_1 + 1$,
  and let $z \in  Z_1$ be general.
  From \autoref{thm:join-sweepZ},
  $Z_1$ is swept out by $Z_1 \cap \TT_xX$'s with general $x \in  X_z$.
  Take a general $w \in  Z_1 \cap \TT_xX$.
  Then $w$ is general in  $Z_1$;
  in  particular, $w \notin Z_2$.
  From \autoref{thm:M-in-Xz}, we have
  $M = \lin{wxz} \subset X_z = \Cone_z(Z_2)$.
  Then we have a curve $A \subset Z_2$
  such that $M = \Cone_z(A)$.

  Let $l_1, l_2 \in  M = \PP^2$ be two general lines passing through $z$.
  Take a point $\alpha_1 \in  l_1 \cap A$.
  Since the point $\alpha_1$ runs over $A$ when $l_1$ sweeps $M$, 
  the line $\lin{\alpha_1 w}$ also sweeps $M$.
  In particular, we may take $x_2 \in  \lin{\alpha_1 w} \cap l_2$ which is general in  $l_2$.
  Also take a general point $x_1 \in  l_1$, and a point $\alpha_2 \in  l_2 \cap A$.
  In this setting, Terracini's lemma implies
  $\TT_{x_i}X = \lin{\TT_{z}Z_1, \TT_{\alpha_i}Z_2}$.

  Since $w \notin A$,
  we also have $M = \Cone_w(A)$,
  which is contained in  $\Cone_{w}(Z_2)$.
  Then
  $\lin{\TT_{\alpha_1}Z_2, w}
  = \TT_{x_2} \Cone_{w}(Z_2) \subset \TT_{x_2}X$.
  Thus $\TT_{\alpha_1}Z_2 \subset \TT_{x_2}X$, and then
  \[
    \TT_{x_2}X = \lin{\TT_{z}Z_1, \TT_{\alpha_1}Z_2} = \TT_{x_1}X.
  \]
  It follows that
  $\lin{zx_1} \cup \lin{zx_2} \subset \fib{x_1}$,
  which contradicts that a general fiber of $\gamma$ is a line (in  particular, is irreducible).

  Hence $\dim (\lin{Z_i}) = \dim(Z_i) + 1$ for $i=1,2$.
  Since $X$ is a cubic hypersurface,
  we have $L := \lin{Z_1} = \Sec(Z_1) \subset X$.
  Note that $L \not\subset \Sing(X)$ (otherwise, $X = \Cone_{L}(Z_2)$, a contradiction).
  Let $y \in  L \setminus \Sing(X)$.
  
  For general $x \in  X \setminus \TT_yX$, taking
  $K := \lin{L,x} \subset \PN$, we consider
  $X \cap K = L \cup X'$,
  where $X'$ is a quadric hypersurface of $K$ because of $L \not\subset \Sing(X \cap K)$.
  Since
  \[
    Z_1 \subset \Sing(X \cap K) = (L \cap X') \cup \Sing(X') \subset X',
  \]
  we have
  $Z_1 \subset L \cap X'$.
  Thus $Z_1$ is a quadric hypersurface of $L$. 
  Then $Z_1$ is smooth
  (otherwise, since $Z_1$ is a cone, so is the join $X$, a contradiction).

  Now we study the intersection of quadrics $Z_1$ and $Z_2$.
  Since $\delta_X=1$, it follows $\dim Z_1 + \dim Z_2 + 1 = N-1$.
  Since $\lin{Z_1, Z_2} = \PN$,
  two linear subvarieties $\lin{Z_1}, \lin{Z_2} \subset \PN$
  intersect only at a point $z_0 \in  \PN$.
  If $z_0 \notin Z_1$, then $Z_1 \cap \lin{Z_2} = \emptyset$
  implies $X = \Join(Z_1, Z_2) = \bigcup_{z \in  Z_1} \Cone_z(Z_2)$,
  where we need not take the closure of the union;
  in  particular,
  since $\lin{Z_2} \cap \Cone_z(Z_2) = Z_2$ for all $z \in  Z_1$, it follows
  $\lin{Z_2} \not\subset X$, a contradiction.
  Hence $z_0 \in  Z_1$. In the same way, $z_0 \in  Z_2$.
  Therefore the assertion follows.
\end{proof}

\begin{ex}\label{thm:calc-Join}
  Let $X = \Join(Q_1, Q_2)$ as in  (II) of \autoref{mainthm}.
  Then $X$ is indeed a cubic hypersurface as follows.

  Set $p = \dim Q_1$ and $q = \dim Q_2$.
  Then $N = p+q+2$.
  Choosing the homogeneous coordinates
  $\pco{x_0: x_1: \dots: x_p: x_{p+1}: \dots: x_{p+q}: y_1: y_2}$
  on $\PP^{p+q+2}$,
  we may assume
  \begin{gather*}
    Q_1 = (-x_0y_1+x_1^2+\dots+x_p^2 = 0) \ \text{ in  }\ 
    \PP^{p+1} = (x_{p+1} = \dots = x_{p+q} = y_2 = 0),
    \\
    Q_2 = (-x_0y_2+x_{p+1}^2+\dots+x_{p+q}^2 = 0) \ \text{ in  }\ 
    \PP^{q+1} = (x_{1} = \dots = x_{p} = y_1 = 0),
  \end{gather*}
  where $Q_1 \cap Q_2 = \PP^{p+1} \cap \PP^{q+1} = \Set{(1,0,\dots,0)}$.
  Then $X \subset \PP^{p+q+2}$
  is defined by a cubic homogeneous polynomial
  \[
    -x_0y_1y_2 + y_1(x_{p+1}^2+\dots + x_{p+q}^2) + y_2(x_1^2+\dots+x_p^2).
  \]
  This is because, the join $X$ is equal to the closure of
  the image of the morphism
  $\A^p \times \A^q \times \A^1 \rightarrow \PP^{p+q+2}$,
  which is given by
  \begin{multline*}
    ((a_1, \dots, a_p), (b_1, \dots, b_q), c) \\ \mapsto 
    \pco{1+c: a_1: \dots: a_p: cb_1: \dots: cb_q: a_1^2+\dots+a_p^2: c(b_1^2+\dots+b_q^2)}.
  \end{multline*}
  Note that $\Sing(X) \subset (y_1 = 0) \cup (y_2 = 0)$.
\end{ex}

\begin{rem}\label{rem:join-two-varieties}
  In the above example,
  the assumption of \autoref{thm:M-in-Xz} does not satisfied.
  Consider $Z_1 = Q_1$.
  Then for general $z \in  Q_1$,
  $X_z = \Cone_{z}(Q_2)$.
  For general $x \in  X_z$, we have
  $\lin{Q_1} \cap \TT_{x}X = \TT_{z}Q_1$,
  and hence
  $Q_1 \cap \TT_{x}X = Q_1 \cap \TT_{z}Q_1$.
  Thus $Q_1$ cannot be swept out by $Q_1 \cap \TT_{x}X$'s with general $x \in  X_z$.

  Note that $Q_1 \subset \TT_xX$ for a special point $x$ of $X_z$; indeed,
  for $s := Q_1 \cap Q_2$ and $x \in \lin{sz} \subset X_z$,
  we have $\lin{Q_1} \subset \TT_xX$.
\end{rem}

\begin{rem}\label{thm:Y-sweptout-gen-members}
  Let $Y \subset \PN$ be a closed variety,
  let $A \subset \Pv$ be a curve, and let $A\spcirc \subset A$ be a non-empty open subset.
  Assume that $Y \not\subset H$ for any $H \in A$.
  Then $Y \subset \PN$ is swept out by $Y \cap H$ with $H \in A\spcirc$.
  This is because, $\bigcup_{H \in  A} H = \PN$,
  and $Y$ is not contained in a union of finitely many members of $A$.
\end{rem}

\begin{rem}\label{thm:del1-I-II}
  Let $X$ be a cubic hypersurface with $\delta_X = 1$, which is not a cone.
  Assume $X = \Sec(\Sing(X))$,
  and let $S_1, \dots, S_{r}$ be the irreducible components of $\Sing(X)$.
  For $x,y \in \Sing(X)$, since $x \in S_i$ and $y \in S_j$ with some $1 \leq i,j \leq r$,
  we have $X = \bigcup_{1 \leq i \neq j \leq r} \Join(S_i, S_j) \cup \bigcup_{1 \leq i \leq r} \Sec(S_i)$, where the joins and secants are irreducible varieties.
  Then (I) of \autoref{mainthm} means that $X = \Sec(S_i)$ for some $i$.
  If $X = \Join(S_i, S_j)$ for some $i \neq j$,
  then \autoref{thm:char-Xjoin} implies the condition~(II) of \autoref{mainthm}.
\end{rem}

\section{Locus of intersection points of \\ the singular locus and general contact loci}
\label{sec:secant-sing-locus}

In this section, we study the case when $\delta_X=1$ and $\Sec(\Sing(X)) \neq X$.
Under the condition, it follows $\kappa = 1$ as in \autoref{thm:def-Z} and \autoref{thm:xi-2-case}.
We write $Z = Z_1$.
Our goal is to show the following result.

\begin{thm}\label{thm:large-lin-Z}
  Let $X \subset \PN$ be a cubic hypersurface with $\delta_X = 1$ which is not a cone.
  Assume that $\Sec(\Sing(X)) \neq X$, and take $Z = Z_1 \subset \Sing(X)$.
  Then $\lin{Z} \subset X$.
\end{thm}

Recall that, for the closure $F \subset X$
of a general fiber of $\gamma$,
it holds $F \cap \Sing(X) = F \cap Z = \set{z}$ with some $z \in  Z$.
In this case,
we have the following statement.

\begin{lem}\label{thm:Xz-in-TTyX-for-y-sz}
  Let $X \subset \PN$ be a cubic hypersurface,
  and let $F \subset X$ be a line contracted to a point under $\gamma$
  such that $F \cap \Sing(X) = \set{z}$.
  Let $s \in  \Sing(X)$ be a point such that $\lin{sz} \not\subset \Sing(X)$.
  Then $F \subset \TT_yX$ for $y \in  \lin{sz} \setminus \Sing(X)$.
\end{lem}
\begin{proof}
  Set $M = \lin{sF} \subset \PN$. If $M \subset X$,
  we have $F \subset M \subset \TT_yX$ immediately.
  Assume $M \not\subset X$.
  Then $\lin{sz} \cup F \subset X \cap M$.

  We have $\Sing(X \cap M) = \lin{sz}$ and $\lin{sz} \cup F = X \cap M$.
  Otherwise,
  as in  \autoref{thm:3rd-line},
  it holds that $X \cap M = \lin{sz} \cup F \cup l$
  and $\Sing(X \cap M) = \Set{w, s, F \cap l}$,
  where $l \subset M$ is a line passing through $s$ such that $l \neq \lin{sz}$.
  Let $F \cap l = \set{x_0}$, where $x_0 \notin \Sing(X)$ by assumption.
  Since $F$ is contracted to a point under $\gamma$, we have
  $M \subset \TT_{x_0}X = \TT_{x}X$ for general $x \in  F$,
  that is, $F \subset \Sing(X \cap M)$, a contradiction.
  Therefore $\Sing(X \cap M) = \lin{sz}$, yielding $F \subset M \subset \TT_yX$.
\end{proof}

\subsection{Secant line of the singular locus}
\label{sec:gauss-map-cone}

In order to show \autoref{thm:large-lin-Z},
we argue by contradiction
and suppose $\lin{Z} \not\subset X$.
In this subsection, we show the following proposition.

\begin{prop}\label{thm:dim-gamma-ConezZ}
  Let $X \subset \PN$ be as the assumption of \autoref{thm:large-lin-Z},
  and assume $\lin{Z} \not\subset X$.
  Let $w \in  Z$ be a general point.
  Then
  $\Cone_w(Z) \subset X$ is not contained in  $\Sing(X \cap \lin{Z})$,
  hence, is not contained in  $\Sing(X)$.

\end{prop}

\begin{lem}\label{thm:XP'}
  For a linear variety $P' \subset \PN$ containing $Z$,
  the intersection
  $X \cap P'$ is an irreducible cubic hypersurface of $P'$.
\end{lem}
\begin{proof}
  We have $Z \subset X \cap P'$.
  If $X \cap P'$ is set-theoretically equal to union of hyperplanes $L_i$ of $P'$,
  then $Z \subset L_i$ and then $\lin{Z} \subset L_i \subset X$, a contradiction.
  Suppose $X \cap P' = L \cup Q$ for a hyperplane $L$ and a quadric hypersurface $Q$ of $P'$.
  Then $Z \subset \Sing(X \cap P') = (L \cap Q) \cup \Sing(Q)$.
  Since $Z \not\subset L$, we have $Z \subset \Sing(Q)$.
  Since $Q$ is a quadric,
  the locus $\Sing(Q)$ is the vertex of the cone,
  in  particular, is a linear variety.
  Then $\lin{Z} \subset \Sing(Q) \subset X$, a contradiction.
  Hence $X \cap P'$ is an irreducible cubic.
\end{proof}

\begin{lem}\label{thm:Z-swout-T-Xz}
  Let $z \in  Z$ be a general point.
  Then $Z$ is swept out by $Z \cap \TT_xX$'s with general $x \in  X_z$,
  where $X_z \subset X$ is a cone consisting of general fibers of $\gamma$ passing through $z$
  as in \autoref{thm:def-Z}.
\end{lem}
\begin{proof}
  Set $r = \dim Z$
  and $P = \lin{Z} \subset \PN$.
  From \autoref{thm:singX-N-2}, we may assume $r \leq N-3$.
  From $P \not\subset X$, it follows that
  \[
    Z \subsetneq \Sec(Z) \subsetneq P;
  \]
  thus $\dim P \geq r+2$.
  Suppose $\dim P = r+2$. Then
  $\dim (\Sec(Z)) = r+1$.
  It follows from \autoref{thm:sec-linear}
  that $\Sec(Z)$ is a linear variety, i.e.,
  $\Sec(Z) = P$, a contradiction.
  Hence $\dim P \geq r+3$.

  On the other hand, $\dim (\gamma(X_z)) = N-r-2 \geq 1$.
  Since $\dim(P^*) \leq N-r-4$,
  the intersection $\overline{\gamma(X_z)} \cap P^* \subset \Pv$
  is of codimension $2$ in  $\overline{\gamma(X_z)}$.
  (Note that $P^* = \emptyset$ if $P = \PN$.)
  We may take a curve $A \subset \overline{\gamma(X_z)}$ not intersecting $P^*$.
  Then, as in \autoref{thm:Y-sweptout-gen-members},
  the locus $Z$ is swept out by $Z \cap H$'s with general $H \in  A$.
  Hence the assertion follows.
\end{proof}

We take the $(N-2)$-dimensional cone $C_w \subset X$ and
$\hat{\mathscr{C}} \subset X\spcirc \times Z$ as in \autoref{thm:Cw},
and consider the following set,
\[
  \mathscr{B} := \set*{(u, w, z) \in  X\spcirc \times Z \times Z}{w \in  \TT_uX,\, \fib{u} \cap \Sing(X) = \set{z}},
\]
with the projection $\rho_i$ from $\mathscr{B}$ to the $i$-th factor,
where $\mathscr{B} \simeq \hat{\mathscr{C}}$.
\begin{cor}\label{thm:Z-in-Cw}
  Let $w \in  Z$ be a general point.
  Then $\rho_3(\rho_2^{-1}(w)) \subset Z$ is a dense subset.
  Hence $Z \subset C_w$ and $\lin{C_w} = \PN$.
\end{cor}
\begin{proof}
  From \autoref{thm:Z-swout-T-Xz}, the projection to the second and third factors,
  $(\rho_2,\rho_3): \mathscr{B}\spcirc \rightarrow Z \times Z$ is dominant.
  Hence $\rho_3(\rho_2^{-1}(w))$ is dense in  $Z$.
  Since $\fib{u} \subset C_w$ for general $u \in  C_w$,
  it follows $Z \subset C_w$.

  We have $\deg(C_w) > 1$;
  otherwise, $\lin{Z} \subset C_w \subset X$, contrary to our assumption.
  Suppose that there exists a hyperplane $P' \subset \PN$
  such that $C_w \subset P'$.
  Then $C_w = X \cap P'$ because of \autoref{thm:XP'}.
  Since $Z \subset P'$, we have
  $X \cap \lin{Z} = X \cap P' \cap \lin{Z} = C_w \cap \lin{Z}$,
  which is a cone with vertex $w$. Since $w$ is general,
  it is a cone with vertex $\lin{Z}$, a contradiction.
  Thus $\lin{C_w} = \PN$.
\end{proof}

\begin{proof}[Proof of \autoref{thm:dim-gamma-ConezZ}]
  Take a general $w \in  Z$.
  From \autoref{thm:Z-in-Cw}, we have $Z \subset C_w$.
  Let $P:=\lin{Z} \subset \PN$, and
  let $C_w^P$ be an irreducible component of $C_w \cap P$ containing $Z$,
  where $\dim C_w^P \geq \dim P - 2$.
  Then $C_{w_1}^P \neq C_{w_2}^P$ for general $w_1 \neq w_2 \in  Z$
  (otherwise, $C_w^P$ is a cone with vertex $P=\lin{Z}$, which is absurd).

  Thus $X^P := X \cap P$ is swept out by $C_w^P$'s with general $w \in  Z$.
  Let $\gamma_{X^P}: X^P \dashrightarrow P\spcheck$
  be the Gauss map of $X^P \subset P$,
  where $P\spcheck = \Gr(\dim P-1, P)$ is the set of hyperplanes of $P$.

  \begin{claim}
    Let $w \in Z$ be general.
    Assume that
    $\lin{uw} \cap \Sing(X^P) = \set{w}$
    for general $u \in  C_w^P$.
    Then $\gamma_{X^P}(\lin{uw}) \subset P\spcheck$ is of dimension $1$.
  \end{claim}
  \begin{proof}
    Suppose that $\lin{uw} \subset X^P$ is contracted to a point
    under $\gamma_{X^P}$.
    Let $S^P$ be the irreducible component of $\Sing(X^P)$ containing $Z$.
    Then we have $\Cone_{w}(S^P) \not\subset \Sing(X^P)$
    (otherwise, we have $\Cone_w(S^P) = S^P$, which means that
    $S^P$ is a cone with vertex $w$; since $w \in  Z$ is general,
    it follows that $S^P$ is a cone with vertex $P = \lin{Z}$,
    a contradiction).

    For general $s \in  S^P$,
    and for $y \in  \lin{sw} \setminus \Sing(X^P)$,
    it follows that $C_w^P \subset \TT_yX^P$
    as in  \autoref{thm:Xz-in-TTyX-for-y-sz}.
    On the other hand,
    the image $\gamma_{X^P}(\Cone_{w}(S^P)) \subset P\spcheck$ is of dimension $\geq 1$
    (otherwise, the image is the set of a point $T \in  P\spcheck$; then
    since $\lin{sw} \subset \TT_yX^P = T$, we have $Z \subset S^P \subset T$,
    contradicting $\lin{Z} = P$).
    Therefore taking $T, T' \in  \gamma_{X^P}(\Cone_{w}(S^P))$,
    we have $C_w^P = T \cap T'$, which is a linear variety.
    Then $P = \lin{Z} \subset C_w^P$, a contradiction.
  \end{proof}

  Now suppose $\lin{ww'} \subset \Sing(X^P)$ for general $w, w' \in  Z$.
  For general $u \in  C_w^P$, take the $2$-plane $M := \lin{uww'} \subset P$.

  If $\lin{uw} \cap \Sing(X^P) \neq \set{w}$,
  then $M \subset X^P$
  (otherwise, for $\tilde{u} \in  \lin{uw} \cap \Sing(X^P)$ with $\tilde{u} \neq w$,
  we have $\set{\tilde{u}} \cup \lin{ww'} \subset \Sing(X \cap M)$,
  contrary to \autoref{thm:3rd-line}).
  If $\lin{uw} \cap \Sing(X^P) = \set{w}$,
  then by the above claim,
  the locus $Z$ is swept out by $\TT_{\tilde{u}}X^P$'s with general $\tilde{u} \in  \lin{uw}$;
  in  particular, we may take a point $\tilde{u} \in  \lin{uw}$ such that $w' \in  \TT_{\tilde{u}}X$,
  and then $M \subset \TT_{\tilde{u}}X$.
  Hence $M \subset X^P$.

  As a result, $w' \in  M \subset \TT_{u}X^P$. Since $w' \in  Z$ is general,
  it follows that $P = \lin{Z} \subset \TT_{u}X^P$, a contradiction.
\end{proof}

Next we give the following two corollaries.

\begin{cor}\label{thm:dim-gamma-cone-zZ}
  Let $w \in  Z$ be general.
  Then $Z \not\subset T$ for general $T \in  \gamma(\Cone_w(Z))$.
  Hence $\dim (\gamma(\Cone_w(Z))) \geq 1$,
  and $C_w$ is swept out by $C_w \cap T$'s with general $T \in  \gamma(\Cone_w(Z))$.
\end{cor}
\begin{proof}
  Suppose $Z \subset \TT_yX$ for general $y \in  \Cone_w(Z)$.
  Then $y \in  \Sing(X \cap \lin{Z})$, and hence
  $\Cone_{w}(Z) \subset \Sing(X \cap \lin{Z})$,
  contradicting \autoref{thm:dim-gamma-ConezZ}.
  Therefore $Z \not\subset \TT_yX$.

  Suppose $\gamma(\Cone_w(Z)) = \Set{T}$. 
  Then $\lin{ww'} \subset T = \TT_yX$ for general $w' \in  Z$ and $y \in  \lin{ww'}$,
  which implies $Z \subset T = \TT_yX$, a contradiction.
\end{proof}

\begin{cor}\label{thm:Z-sw-out-ZnT}
  Let $u \in  C_w$ be general. Then
  $Z$ is swept out by $\TT_{u_0}X$'s for general $u_0 \in  \lin{uw}$.
\end{cor}
\begin{proof}
  Note that $\lin{uw} \cap \Sing(X) = \set{w}$ since $\Sec(\Sing(X)) \neq X$.
  Let $w' \in  Z$ be general,
  and take $y \in  \lin{ww'} \setminus \Sing(X)$.
  We may assume $u \in  C_w \setminus \TT_yX$.
  For $M := \lin{uww'} \subset \PN$.
  it follows
  $M \not\subset X$ (otherwise, $u \in  M \subset \TT_yX$, a contradiction).
  Since $\lin{ww'} \not\subset \Sing(X \cap M)$,
  it follows from \autoref{thm:3rd-line}
  that
  there exists a line $l \subset X \cap M$ with $l \neq \lin{ww'}$ passing through $w'$.
  For the intersection point $u_0 = \lin{uw} \cap l$,
  we have $w' \in  l \subset \TT_{u_0}X$.
  This implies the assertion.
\end{proof}

\subsection{Cones of codimensions one in  the cubic hypersurface}
\label{sec:cones-codim-one}

Using the tangent hyperplane $T \in  \gamma(\Cone_w(Z)) \subset \Pv$
discussed in  the previous subsection,
we investigate the structure of the cone $C_w$ of codimension one in  $X$,
and give the proof of \autoref{thm:large-lin-Z}.

\begin{proof}[Proof of \autoref{thm:large-lin-Z}]
  Assume $X \neq \Sec(\Sing(X))$ and write $Z = Z_1$ as above.
  Assume $\lin{Z} \not\subset X$.
  Take general points $w, w' \in  Z$
  and set $T = \TT_yX$ for $y \in  \lin{ww'} \setminus \Sing(X)$.
  Note that $\lin{ww'} \subset T$.
  From \autoref{thm:dim-gamma-cone-zZ}, we have $Z \not\subset T$.

  Since $\Sec(\Sing(X)) \neq X$,
  we may take a non-empty open subset $X\spcirc \subset X$ consisting of smooth points $x$
  such that $\fib{x} \cap \Sing(X) = \set{z}$ with some $z \in  Z$.

  \begin{claim}\label{thm:gammaCz-n-T}
    $\gamma(C_w \cap T) = \gamma(C_w)$ for general $w \in Z$.
  \end{claim}
  \begin{proof}
    Take $F = \fib{x} \subset C_w$ for
    general $x \in \gamma(C_w) \cap X\spcirc$.
    Then $z = F \cap \Sing(X)$ is a general point of $Z$
    since $\rho_3(\rho_2^{-1}(w)) \subset Z$ is dense
    as in  \autoref{thm:Z-in-Cw}.
    In particular, $z \notin T$.
    Since $F \cap T \notin \Sing(X)$, we have $\gamma(F \cap T) = \gamma(F) \in  \Pv$.
    Thus $\gamma(C_w \cap T) = \gamma(C_w)$.
  \end{proof}

  From \autoref{thm:Z-sw-out-ZnT},
  we may take an open subset $C_w\spcirc \subset C_w \cap X\spcirc$ satisfying
  that,
  for each $u \in  C_w\spcirc$,
  there exists
  $u_0 \in  \lin{uw} \cap X\spcirc$ such that $w' \in  \TT_{u_0}X$.
  Let $t \in  C_w \cap T$ be general.
  Note that $\lin{tw} \subset T$ since $w \in  T$.

  \begin{claim}
    There exists $t_0 \in  \lin{tw}$ such that $t_0 \neq w$ and $w' \in  \TT_{t_0}X$.
  \end{claim}
  \begin{proof}
    Since $\gamma(C_w) = \gamma(C_w \cap T)$,
    the variety $C_w$ is swept out by fibers
    $\fib{t}$'s with general $t \in  C_w \cap T$.
    Then we may take $u \in  \fib{t} \cap C_w\spcirc$,
    and take $u_0 \in  \lin{uw} \cap X\spcirc$ such that $w' \in  \TT_{u_0}X$.
    Let $z \in  Z$ be the intersection point of $\fib{u}=\fib{t}$ and $Z$.
    Since $w \in  Z \cap \TT_uX$,
    it follows from \autoref{thm:M-in-Xz} and \autoref{thm:Z-swout-T-Xz} that $\lin{u_0z} = \fib{u_0}$.
    Note that $\lin{u_0z} \cap \Sing(X) = \set{z}$ since $u_0 \in  X\spcirc$.
    For the intersection point $t_0$ of $\lin{tw}$ and $\lin{u_0z}$,
    we have $w' \in  \TT_{t_0}X$.
  \end{proof}

  Set $M = \lin{t_0ww'}$.
  Then $M \subset X$
  (otherwise, $y, t_0, w, w' \in  \Sing(X \cap M)$, contrary to \autoref{thm:3rd-line}).
  In particular, $w' \in  \TT_{t}X$.
  Hence $\gamma(C_w \cap T) \subset {w'}^*$ in  $\Pv$.
  Since $\gamma(C_w) = \gamma(C_w \cap T)$,
  we have $\gamma(C_w) \subset {w'}^*$.
  Since $w \in  Z$ is general, 
  we have $X^* = \gamma(X) \subset {w'}^*$, 
  which means that $X$ is a cone with vertex $w'$,
  a contradiction.
\end{proof}

In addition, the following statement holds.

\begin{cor}\label{thm:Z-codim1-span}
  In the setting of \autoref{thm:large-lin-Z},
  if $Z$ is of codimension $\leq 1$ in $\lin{Z}$,
  then $\lin{Z} \subset \Sing(X)$.
\end{cor}
\begin{proof}
  Otherwise, for general $y \in  \lin{Z} \setminus \Sing(X)$ and $z \in  Z$,
  since $\deg(Z) > 1$, we have a point $z' \in  \lin{yz} \cap Z$ with $z \neq z'$.
  By \autoref{thm:Xz-in-TTyX-for-y-sz}, we have $X_z \subset \TT_yX$.
  Since $z$ is general, it follows $X \subset \TT_yX$, a contradiction.
\end{proof}

\section{Induction on dual defects}
\label{sec:induct-dual-defects}

Let $X \subset \PN$ be a cubic hypersurface with $\delta_X > 0$, which is not a cone.
Let $\gamma: X \dashrightarrow X^* \subset \PN$ be the Gauss map of $X$.

If $\Sec(\Sing(X)) \neq X$,
then it follows from \autoref{thm:FnSing-codim1} that
the intersection $\fib{x} \cap \Sing(X)$ is a $(\delta_X-1)$-plane
for general $x \in X$.
Let $X\spcirc \subset X$ be a non-empty open subset which consists of $x \in X$
satisfying the above condition. We consider
\[
  \set*{(x, z) \in X\spcirc \times \Sing(X)}{z \in \fib{x}} \rightarrow \Sing(X),
\]
The left hand side is irreducible, and so is
the closure $Z = Z_X \subset \Sing(X)$ of the image of the projection.
Indeed, $Z$ is the closure of the union of
$\fib{x} \cap \Sing(X)$'s with general $x \in X$,
and that $\fib{x} \cap \Sing(X) = \fib{x} \cap Z$.
Note that $\dim Z \geq \delta_X$ if $X$ is not a cone.
In this setting, we consider the following condition,
\begin{enumerate}
\item[\quad (III')\;] $\Sec(\Sing(X)) \neq X$ and $\lin{Z} \subset X$.
\end{enumerate}

Now let us show the proposition below by induction on $\delta_X > 0$,
where (I) and (II) are conditions in  \autoref{mainthm}.

\begin{prop}\label{mainthm-b}
  Let $X \subset \PN$ be a cubic hypersurface with $\delta_X > 0$
  and assume that $X$ is not a cone.
  Then one of the $3$ conditions \textnormal{(I)}, \textnormal{(II)}, and \textnormal{(III')} holds.
\end{prop}

\begin{lem}\label{thm:ZXH}
  Let $X \subset \PN$ be a projective variety with $\delta_X > 0$. 
  Let $H \subset \PN$ be a general hyperplane. Then
  the dual defect of
  $X \cap H \subset H = \PP^{N-1}$
  is $\delta_X - 1$.
  In addition, if $X$ is not a cone, so is $X \cap H$.
\end{lem}

Let $x \in X$ be a general point.
For a hyperplane $H \subset \PN$ containing $x$ with $H \neq \TT_xX$,
we set ${X'} := X \cap H \subset H=\PP^{N-1}$.
Then $\TT_x{X'} = \TT_xX \cap H$.
Since
\begin{equation*}
  \overline{\gamma_X^{-1}(\gamma_X(x))} \cap H
  \subset \overline{\gamma_{X'}^{-1}(\gamma_{X'}(x))}.
\end{equation*}
it immediately follows $\delta_{X'} \geq \delta_X - 1$.
We have the diagram
\begin{equation}\label{eq:gamma-XH}
  \xymatrix{
    X \ar@{-->}[r]^{\gamma_X} & X^* \ar[d] \ar@{}[r]|{\mbox{$\subset$}}
    & \Pv \ar[d]^{\pi_{[H]}}
    \\
    {X'} \ar@{-->}[r]^{\gamma_{X'}} \ar@{}[u]|{\rotatebox{90}{$\subset$}} & \pi_{[H]}(X^*) \ar@{}[r]|{\mbox{$\subset$}}
    & \Pv[N-1],
  }    
\end{equation}
where $\pi_{[H]}$ is the linear projection
from the point $[H] \in  \Pv$;
note that $\pi_{[H]}([L]) = [L \cap H]$ for $[L] \in  \Pv$.

\begin{proof}[Proof of \autoref{thm:ZXH}]
  Let $[H] \in  \Pv$ be general. Then $\pi_{[H]}|_{X^*}$ is birational.
  For a general fiber $F$ of $\gamma_X$, since $\dim F = \delta_X > 0$,
  we have $F \cap H \neq \emptyset$;
  hence $\gamma_X({X'}) = X^*$.
  Therefore ${X'}^* = \gamma_{X'}({X'}) = \pi_{[H]}(X^*)$, which implies $\delta_{X'} = \delta_X - 1$.

  Next, suppose that ${X'} := X \cap H$ is a cone with vertex $v$.
  Since $F \cap H$ is a general fiber of $\gamma_{X'}$,
  we have $v \in  F \cap H \subset F$.
  Since $F$ is a linear variety,
  we find that $X$ is a cone with vertex $v$.
\end{proof}

\begin{proof}[Proof of \autoref{mainthm-b}]
  We use induction on $\delta_X > 0$.
  The case of $\delta_X=1$ follows from \autoref{thm:char-Xjoin}, \autoref{thm:del1-I-II}, and \autoref{thm:large-lin-Z}.
  Let $X \subset \PN$ be a cubic hypersurface with $\delta_X > 1$ which is not a cone.

  For general $H \in  \Pv \setminus X^*$,
  the intersection $X \cap H \subset H = \PP^{N-1}$
  is irreducible and $\Sing(X \cap H) = \Sing(X) \cap H$.
  Let $S_1, \dots, S_{r}$ be the irreducible components of $\Sing(X)$.
  From \autoref{thm:ZXH} and by hypothesis,
  $X \cap H$ satisfies (I), (II), or (III').
  \vspc
  
  Assume that $X \cap H$ satisfies (I).
  Then $X \cap H = \Sec(S_i \cap H)$ for some $i$.
  Indeed, we may take an index $i_0$
  such that $X \cap H = \Sec(S_{i_0} \cap H)$ for general $H \in  \Pv$.
  Then $X = \Sec(S_{i_0})$, i.e., $X$ also satisfies the condition (I).
  \vspc
  
  Assume that $X \cap H$ satisfies (II).
  Then we may take indices $i_1, i_2$
  such that $S_{i_1} \cap H$ and $S_{i_2} \cap H$ are quadrics
  appeared in  the condition~(II) for general $H \in  \Pv$.
  For simplicity, let $(i_1, i_2) = (1,2)$.
  Then $X = \Join(S_1, S_2)$ and $S_j$ is a smooth quadric hypersurface
  of $\lin{S_j} = \PP^{\dim S_j+1}$ with $j = 1,2$.
  In addition, $L := \lin{S_1} \cap \lin{S_2}$ is a line since
  $\lin{S_1} \cap \lin{S_2} \cap H$
  is the set of a point
  for general $H$.

  Let $x \in  X$ be a general point
  with $x \in  \lin{ab}$ for $a \in  S_1$ and $b \in  S_2$.
  Since $\TT_aS_1$ is a hyperplane of $\lin{S_1}$,
  and since $\TT_bS_2 \cap L \neq \emptyset$,
  we have $\lin{S_1} \subset \lin{\TT_aS_1, \TT_bS_2} = \TT_xX$.
  In the same way,
  $\lin{S_2} \subset \TT_xX$.
  It follows $\PN = \TT_xX$, a contradiction. Thus the case does not occur for $\delta_X > 1$.
  \vspc

  Finally, assume that $X \cap H$ satisfies (III').
  We take $Z_{X \cap H} \subset \Sing(X \cap H)$.
  Then
  \[
    F' \cap \Sing(X \cap H) = F' \cap Z_{X \cap H}
  \]
  for the fiber
  $F' := \overline{\gamma_{X\cap H}^{-1}(\gamma_{X\cap H}(x))}$ at a general $x \in X \cap H$.
  Since $F' = \fib{x} \cap H$ and since $H$ is general,
  we find that $\fib{x} \cap \Sing(X)$ is a linear variety.
  In particular, $\Sec(\Sing(X)) \neq X$ (see \autoref{thm:F-is-sec} below).

  We take $Z = Z_X \subset \Sing(X)$.
  Since
  $F' \cap \Sing(X\cap H) = \fib{x} \cap \Sing(X) \cap H$,
  we have $Z \cap H = Z_{X \cap H}$.
  In addition,
  $\lin{Z} \cap H = \lin{Z_{X \cap H}}$
  (this is because, if $\lin{A} = \PP^m$ for a variety $A$,
  then $\lin{a_0, \dots, a_{m-1}} = \PP^{m-1}$ for general ${a_0, \dots, a_{m-1}} \in A$).
  Since $\lin{Z_{X \cap H}} \subset X \cap H$ for general $H$,
  we have $\lin{Z} \subset X$.
\end{proof}

\begin{rem}\label{thm:F-is-sec}
  Assume $X = \Sec(\Sing(X))$. Let $F \subset X$ be the closure of a general fiber of $\gamma$.
  Taking a general $x \in X$, since $x \in \lin{s_1s_2}$ with some $s_1, s_2 \in \Sing(X)$,
  we have $\lin{s_1s_2} \subset F$ by Terracini's lemma.
  Hence $F = \Sec(F \cap \Sing(X))$ and that $F \cap \Sing(X)$ is non-linear.

  This implies that if $X$ satisfies the condition~(III) of \autoref{mainthm},
  then $X \neq \Sec(\Sing(X))$.
\end{rem}

\begin{lem}\label{mainthm-b-sub}
  In the case of \textnormal{(III')} of \autoref{mainthm-b},
  the locus $Z$ cannot be a $\delta_X$-plane.
\end{lem}

\begin{proof}
  First we assume $\delta_X=1$,
  and show the assertion by induction on $N$.
  The case of $N = 3$ follows from \autoref{thm:N3}.
  Assume $N > 3$ and suppose that $Z$ is a line.
  Let $H \subset \PN$ be a general hyperplane containing $Z$,
  and let $X' = X'_H$ be an irreducible component of $X \cap H$
  containing a general point $x$ of $X$.
  Since $\dim Z^* = N-2$, 
  we may assume $H \notin X^*$. 
  Then
  \[
    \Sing(X') \subset \Sing(X \cap H) = \Sing(X) \cap H.
  \]
  In particular, $x \notin \Sing(X')$.
  Since $\fib{x} \subset X$ is a line such that $\fib{x} \cap \Sing(X) = \set{z}$ with some $z \in  Z$,
  we have $\fib{x} = \lin{xz} \subset X'$.
  Hence $\delta_{X'} \geq 1$. Moreover we have:
  \begin{claim}
    $\delta_{X'} = 1$.
  \end{claim}
  \begin{proof}
    We consider the diagram \autoref{eq:gamma-XH}.
    It holds $\dim \gamma(X') = N-3$.
    In addition,
    we have $\dim \pi_{[H]}(\gamma(X')) = N-3$;
    otherwise $\overline{\gamma(X')} \subset \Pv$ must be a cone with vertex $[H]$,
    which does not occur since $[H] \notin X^*$.
    Since ${X'}^* = \gamma_{X'}(X') \subset \Pv[N-1]$ is equal to
    $\pi_{[H]}(\gamma(X'))$,
    we have $\delta_{X'} = 1$.
  \end{proof}

  Assume that $X'$ is not a cone. Then $X' = X \cap H$.
  For general $x_1, x_2 \in X'$ with $\fib{x_i} \cap \Sing(X) = \set{z_i}$,
  it follows $z_1 \neq z_2$.
  Hence the line $Z = Z_X$ coincides with $Z_{X'} \subset \Sing(X')$.
  Then we have a contradiction by induction hypothesis.

  Assume that $X' = X'_H$ is a cone with vertex
  $v = v_{H}$. Then $v \in  Z$.
  Since $X$ is not a cone,
  we may assume $v_{H_1} \neq v_{H_2}$ for general $H_1 \neq H_2 \in  Z^*$.
  Hence $v$ is general in  $Z$.
  Let us take $A \subset Z^*$ to be the set of hyperplanes $H \subset \PN$ such that
  $X'_H$ is a cone with vertex $v$.
  Since $\dim A = N-3 \geq 1$,
  the hypersurface $X$ is swept out by $X'_H$ with $H \in A$,
  which implies that $X$ is a cone with vertex $v$, a contradiction.
  Hence the result follows for $\delta_X=1$.

  In the case of $\delta_X > 1$, taking a general hyperplane of $\PN$ as in \autoref{thm:ZXH},
  we have the result by induction on $\delta_X$.
\end{proof}

\begin{proof}[Proof of \autoref{mainthm}]
  The result follows from \autoref{mainthm-b} and \autoref{mainthm-b-sub}.
  In the case of (III'), since $Z$ is of dimension $\geq \delta_X$ and is not a $\delta_X$-plane,
  $\lin{Z} \subset X$ is of dimension $> \delta_X$; hence the condition (III) follows.
\end{proof}

\begin{rem}
  In the case of (III'), if $Z$ is of codimension $\leq 1$ in $\lin{Z}$,
  then $\lin{Z} \subset X$ is indeed contained in $\Sing(X)$.
  This follows from \autoref{thm:Z-codim1-span} and by induction on $\delta_X$.
\end{rem}

\subsection*{Acknowledgments}
The author wish to thank
Professors Hajime Kaji and Hiromichi Takagi
for many helpful comments and advice.
Further, the author would like to thank Professor Francesco Russo
for informing about hypersurfaces with vanishing Hessians.
The author was supported by the Grant-in-Aid for JSPS fellows,
No.~16J00404.


\begin{thebibliography}{MM}

\bibitem{Do}
  I. Dolgachev, Polar Cremona transformations. Mich. Math. J. \textbf{48} (2000) 191--202.

\bibitem{FP}
  G. Fischer and J. Piontkowski, Ruled varieties. Advanced Lectures in  Mathematics, Friedr. Vieweg
  \& Sohn, Braunschweig, 2001.

\bibitem{FH12}
  B.~Fu and J.-M.~Hwang, Classification of non-degenerate projective varieties with non-zero prolongation and application to target rigidity.
  Invent. Math. \textbf{189} (2012) 457--513.

\bibitem{FH18}
  B.~Fu and J.-M.~Hwang, Special birational transformations of type $(2, 1)$.
  J. Alg. Geom. \textbf{27} (2018) 55--89.


\bibitem{GR}
  R.~Gondim and F.~Russo, Cubic hypersurfaces with vanishing hessian,
  J. Pure Appl. Algebra \textbf{219} (2015) 779--806,
  arXiv:1312.1618 (the corrected version).

\bibitem{GRS}
  R.~Gondim, F.~Russo, and G.~Staglian\`{o},
  Hypersurfaces with vanishing hessian via dual cayley trick. arXiv:1802.01295.

\bibitem{Hw}
  J.-M.~Hwang, Prolongations of infinitesimal automorphisms of cubic hypersurfaces with nonzero Hessian. arXiv:1612.09020.

\bibitem{HM}
  J.-M.~Hwang and N.~Mok, Prolongations of infinitesimal linear automorphisms
  of projective varieties and rigidity of rational homogeneous spaces of Picard number 1 under K\"{a}hler deformation.
  Invent. Math. \textbf{160}, 591--645 (2005).

\bibitem{Pe} U. Perazzo, Sulle variet\'{a} cubiche la cui hessiana svanisce identicamente. G. Mat. Battaglini \textbf{38}, 337–354 (1900)

\bibitem{Ru} F.~Russo, On the geometry of some special projective varieties. Lecture Notes
  of the Unione Matematica Italiana, Springer Verlag, Berlin, 2016.

\bibitem{Za}
  F.~L.~Zak, Tangents and secants of algebraic varieties.
  Transl. Math. Monographs {\bf 127},  Amer. Math. Soc., Providence, 1993. 

\end{thebibliography}
\end{document}